\documentclass[12pt,a4paper]{article}
\usepackage{latexsym,amsfonts,amsmath,graphics}
\usepackage{epsfig}
\usepackage{enumitem}

\newtheorem{theorem}{Theorem}

\newenvironment{proof}{\begin{trivlist}
\item[\hskip\labelsep{\it Proof.}]}{$\hfill\Box$\end{trivlist}}

\newcommand{\norm}[1]{\left\Vert#1\right\Vert}

\newcommand{\bsa}{\boldsymbol{a}}

\newcommand{\bsgamma}{\boldsymbol{\gamma}}

\newcommand{\bsb}{\boldsymbol{b}}
\newcommand{\bsx}{\boldsymbol{x}}
\newcommand{\bsh}{\boldsymbol{h}}

\newcommand{\infp}{\operatornamewithlimits{inf\phantom{p}}}
\newcommand{\bsy}{\boldsymbol{y}}

\newcommand{\e}{{\varepsilon}}

\newcommand{\ee}{{\rm e}}

\newcommand{\icomp}{\mathtt{i}}

\newcommand{\rd}{\,\mathrm{d}}
\newcommand{\NN}{\mathbb{N}}
\newcommand{\ZZ}{\mathbb{Z}}

\newcommand{\RR}{\mathbb{R}}
\newcommand{\CC}{\mathbb{C}}

\newcommand{\il}{\left<}
\newcommand{\ir}{\right>}

\newcommand{\abs}[1]{\left\vert#1\right\vert}

\newcommand{\lstd}{\Lambda^{\rm{std}}}
\newcommand{\lall}{\Lambda^{\rm{all}}}

\allowdisplaybreaks

\date{}

\begin{document}

\title{Tractability of multivariate analytic problems}

\author{Peter Kritzer\thanks{P.~Kritzer gratefully
acknowledges the support of the Austrian Science Fund, Project
P23389-N18 and Project F5506-N26, which is part of the Special Research Program ``Quasi-Monte Carlo Methods: Theory
and Applications''.}, Friedrich
Pillichshammer\thanks{F.~Pillichshammer is supported by the Austrian Science Fund (FWF) Project F5509-
N26, which is part of the Special Research Program ``Quasi-Monte Carlo Methods: Theory
and Applications''.}, Henryk Wo\'zniakowski\thanks{H.~Wo\'zniakowski is
supported in part by the National Science Foundation.}}

\maketitle

\begin{abstract}
 In the theory of tractability of multivariate problems one usually
studies problems with finite
smoothness. Then we want to know which $s$-variate problems
can be approximated to within $\varepsilon$ by using, say,
polynomially many in $s$ and $\varepsilon^{-1}$ function
values or arbitrary linear functionals.

There is a recent stream of work for multivariate analytic problems
for which we want to answer the usual tractability
questions with $\varepsilon^{-1}$ replaced by $1+\log \varepsilon^{-1}$.
In this vein of research, multivariate integration and
approximation have been studied over
Korobov spaces with exponentially fast decaying Fourier coefficients.
This is work of J. Dick, G. Larcher,
and the authors. There is a natural need to analyze
more general analytic problems defined over more general spaces
and obtain tractability results in terms of $s$ and $1+\log
\varepsilon^{-1}$.

The goal of this paper is to survey the existing results,
present some new results, and propose
further questions for the study of tractability
of multivariate analytic questions.
\end{abstract}

\noindent\textbf{Keywords:} Tractability, Korobov space, numerical integration,
$L_2$-approximation.\\

\noindent\textbf{2010 MSC:} 65D15, 65D30, 65C05, 11K45.\\

\section{Introduction}

In this paper we discuss algorithms for multivariate integration or
approximation of $s$-variate functions defined
on the unit cube $[0,1]^s$. These problems
have been studied in a large number of papers from many different perspectives.

The focus of this article is to discuss algorithms for
high-dimensional problems defined for functions from
certain Hilbert spaces.
There exist many results for such algorithms, and much progress
has been made on this subject
over the past decades. It is the goal of this review
to focus on a recent vein of research that deals with
function spaces containing
analytic periodic
functions with exponentially fast decaying Fourier coefficients.
We present necessary and sufficient
conditions that allow us
to obtain exponential error convergence and various notions of
tractability.

We consider algorithms that use finitely many information evaluations.
For multivariate integration, algorithms
use $n$ information evaluations
from the class $\Lambda^{\rm{std}}$
of standard information
which consists of only function evaluations.
For multivariate approximation in the $L_2$-norm,
algorithms use $n$ information evaluations either
from the class $\Lambda^{\rm{all}}$ of all continuous linear
functionals or from the class $\Lambda^{\rm{std}}$.
Since we approximate functions from the unit ball of the corresponding
space, without loss of generality we restrict ourselves to linear
algorithms that use nonadaptive information evaluations.
In all cases, we measure the error by considering
the worst-case error setting.
For large~$s$, it is essential to not only control how the error
of an algorithm depends on $n$, but also how it depends on $s$.
To this end, we consider the information complexity, $n(\e,s)$,
which is the minimal number $n$ for which there exists an
algorithm using $n$ information evaluations with
an error of at most $\e$ for the $s$-variate functions.
In all cases considered in this survey,
the information complexity
is proportional to the minimal cost of computing an $\e$-approximation
since linear algorithms are optimal and their implementation
cost is proportional to $n(\e,s)$.

We would like to control how $n(\e,s)$ depends on $\e^{-1}$ and $s$.
This is the subject of tractability. In the standard theory of tractability,
see~\cite{NW08,NW10,NW12}, \emph{weak tractability} means that
$n(\e,s)$ is \emph{not} exponentially dependent  on $\e^{-1}$ and $s$,
\emph{polynomial tractability} means
that $n(\e,s)$ is polynomially bounded in $\e^{-1}$ and $s$,
and \emph{strong polynomial tractability} means that
$n(\e,s)$ is polynomially bounded  in $\e^{-1}$ independently of
$s$.

Typically, $n(\e,s)$ is polynomially dependent on $\e^{-1}$ and $s$
for weighted classes of smooth functions.
The notion of weighted function classes means that
the dependence of functions on successive variables and groups of variables
is moderated by certain weights.
For sufficiently fast decaying weights, the information complexity
depends at most polynomially on $\e^{-1}$ and $s$; hence we obtain
\emph{polynomial} tractability, or even \emph{strong polynomial} tractability.

These notions of tractability are suitable for problems with
finite smoothness, that is, when functions from the problem space
are differentiable only finitely many times.
Then the minimal errors $e(n,s)$ of algorithms that use $n$ information
evaluations typically enjoy polynomial convergence, i.e.,
$e(n,s)=\mathcal{O}(n^{-p})$, where the factor in the big~$\mathcal{O}$
notation as well as a positive $p$ may depend on $s$.

The case of analytic or infinitely many times differentiable functions
is also of interest. For such classes of functions
we would like to replace polynomial convergence by
\textit{exponential convergence}, and study similar notions of
tractability in terms of $(1+\log\,\e^{-1},s)$ instead of
$(\e^{-1},s)$. By exponential convergence we mean that
$e(n,s)=\mathcal{O}(q^{[\mathcal{O}(n)]^p})$
with $q\in(0,1)$, where the factors in the big $\mathcal{O}$
notation as well as a positive $p$ may depend on $s$.

Exponential convergence with various notions of tractability
was studied in the papers~\cite{DLPW11}
and~\cite{KPW12} for multivariate integration in weighted Korobov spaces
with exponentially fast decaying Fourier coefficients.
In the paper \cite{DKPW13}, multivariate $L_2$-approximation
in the worst-case setting for the same class of functions was considered.

In this article, we
give an overview of recent results
on exponential convergence with different
notions of tractability such as weak, polynomial and strong polynomial
tractability in terms of $1+\log\,\e^{-1}$ and $s$.
We also present a few new results and compare conditions which are
needed for the standard and new tractability notions.

In Section~\ref{sectractability}, we give a short overview
of $s$-variate problems, describe
how we measure errors, and give precise
definitions of various notions of tractability.
In Section~\ref{secKor}, we introduce the function class
under consideration here,
which is a special example of a reproducing kernel Hilbert
space that was also studied in \cite{DKPW13, DLPW11, KPW12}.
In Sections~\ref{secint} and~\ref{secapp},
we provide details on the particular problems of $s$-variate
numerical integration and $L_2$-approximation by linear algorithms.
We summarize and give an outlook to some related open
questions in Section~\ref{secconcl}.

\section{Tractability}\label{sectractability}

We consider Hilbert spaces $H_s$ of $s$-variate functions
defined on $[0,1]^s$, and we assume that there is a family of
continuous
linear operators $S_s: H_s\rightarrow G_s$ for $s\in\NN$,
where $G_s$ is a normed space.

Later, we will introduce a special choice of a Hilbert
space $H_s$ (cf. Section~\ref{secKor})
and study two particular examples of $s$-variate problems,
namely:
\begin{itemize}
 \item Numerical integration of functions $f\in H_s$,
see  Section~\ref{secint}. In this case, we have
$S_s (f)=\int_{[0,1]^s} f(\bsx)\rd\bsx$ and $G_s=\RR$.
 \item $L_2$-approximation of functions $f\in H_s$,
see  Section~\ref{secapp}. In this case, we have
 $S_s(f)=f$ and $G_s=L_2 ([0,1]^s)$.
\end{itemize}

As already mentioned, without loss of generality, we
approximate $S_s$ by a linear algorithm $A_{n,s}$
using $n$ information evaluations which are given by linear
functionals from the class $\Lambda\in\{\lall,\lstd\}$. That is,
$$
A_{n,s}(f)=\sum_{j=1}^nL_j(f)\,a_j \ \ \ \ \ \mbox{for all}\ \ \ \ \
f\in H_s,
$$
where $L_j\in \Lambda$ and $a_j\in G_s$ for all $j=1,2,\dots,n$.
For $\Lambda=\lall$ we have $L_j\in H_s^*$ whereas for $\Lambda=\lstd$
we have $L_j(f)=f(x_j)$ for all $f\in H_s$, and for some $x_j\in[0,1]^d$.
For $\lstd$, we choose $H_s$ as a reproducing kernel Hilbert space so
that $\lstd\subset\lall$.

We measure the error of an algorithm $A_{n,s}$
in terms of the \textit{worst-case error}, which is defined as
\[
e(H_s,A_{n,s}):=\sup_{f \in H_s \atop
\|f\|_{H_s} \le 1} \norm{S_s(f)-A_{n,s}(f)}_{G_s},
\]
where $\norm{\cdot}_{H_s}$ denotes the norm in $H_s$,
and $\norm{\cdot}_{G_s}$ denotes the norm in $G_s$. The \textit{$n$th minimal
(worst-case) error} is given by
$$
e(n,s):=\infp_{A_{n,s}} e(H_s, A_{n,s}),
$$
where the infimum is taken over all admissible algorithms $A_{n,s}$.

For $n=0$, we consider algorithms that do not use information
evaluations and therefore we use $A_{0,s}\equiv 0$.
The error  of $A_{0,s}$ is called the \textit{initial (worst-case) error} and
is given by
$$
e(0,s):=\sup_{f \in H_s \atop
\|f\|_{H_s} \le 1} \norm{S_s(f)}_{G_s}=\norm{S_s}.
$$

When studying algorithms $A_{n,s}$, we do not only want
to control how their errors depend on $n$, but also how
they  depend on the dimension $s$. This is of particular
importance for high-dimensional problems. To this end, we define,
for $\e\in(0,1)$ and $s\in\NN$, the {\it information complexity} by
\[
  n(\varepsilon,s):=
  \min\left\{n\,:\,e(n,s)\le\varepsilon \right\}
\]
as the minimal number of information evaluations needed to obtain an
$\e$-approximation to $S_s$. In this case, we speak of
the \textit{absolute error criterion}.
Alternatively, we can also define the information complexity as
\[
  n(\varepsilon,s):=
  \min\left\{n\,:\,e(n,s)\le\varepsilon e(0,s)\right\},
\]
i.e., as the minimal number of information evaluations
needed to reduce the initial error by a factor of $\e$. In this
case we speak of the \textit{normalized error criterion}.

The examples
considered in this paper have the convenient property
that the initial errors are one, and the absolute and normalized
error criteria coincide.
For problems for which the initial errors are not one, the results
for the absolute and normalized error criteria may be quite different;
we refer the interested reader
to the monographs \cite{NW08, NW10, NW12} for further details.

\vskip 1pc

The subject of tractability deals with the question how
the information complexity depends on $\e^{-1}$ and $s$. Roughly speaking,
tractability means that the information complexity lacks
a certain disadvantageous dependence on $\e^{-1}$ and $s$.

The standard notions of tractability were introduced in such a way
that positive results were possible for problems
with finite smoothness. In this case, one is usually
interested in when $n(\e,s)$ depends at most
polynomially on $\e^{-1}$ and $s$.
The following notions have been frequently
studied. We say that we have:

\begin{itemize}
\item[(a)]
The \emph{curse of dimensionality} if there exist positive $c,\tau$ and
$\e_0$ such that
$$
n(\e,s)\ge c\,(1+\tau)^s \ \ \ \ \
\mbox{for all\, $\e\le \e_0$\, and infinitely many $s$}.
$$

\item[(b)] \emph{Weak Tractability (WT)} if
$$
\lim_{s+\e^{-1}\to\infty}\frac{\log\
  n(\varepsilon,s)}{s+\e^{-1}}=0\ \ \ \ \ \mbox{with}\ \ \
 \log\,0=0\ \ \mbox{by convention}.
$$

\item[(c)] \emph{Polynomial Tractability (PT)} if
there exist non-negative numbers $c,\tau_1,\tau_2$ such that
$$
n(\varepsilon,s)\le
c\,s^{\,\tau_1}\,(\e^{-1})^{\,\tau_2}\ \ \ \ \ \mbox{for all}\ \ \ \
s\in\NN, \ \e\in(0,1).
$$
\item[(d)]
\emph{Strong Polynomial Tractability (SPT)} if there exist non-negative
  numbers $c$ and $\tau$  such that
$$
n(\varepsilon,s)\le
c\,(\e^{-1})^{\,\tau}\ \ \ \ \ \mbox{for all}\ \ \ \
s\in\NN, \ \e\in(0,1).
$$
The exponent $\tau^*$ of strong polynomial tractability is defined as
the infimum of $\tau$ for which  strong polynomial tractability holds.
\end{itemize}

It turns out that many multivariate problems defined over
standard spaces of functions suffer from the curse of dimensionality.
The reason for this negative result is that for standard spaces
all variables and groups of variables are equally important.
If we introduce weighted spaces, in which the importance of successive
variables and groups of variables is monitored by corresponding
weights, we can vanquish the curse of dimensionality and obtain weak,
polynomial or even strong polynomial tractability depending on the
decay of the weights.
Furthermore, this holds for weighted spaces with finite smoothness.
We refer to \cite{NW08, NW10, NW12}
for the current state of the art in this field of research.

However, the particular weighted function space
we are going to define in Section~\ref{secKor} is such that
its elements are infinitely many times differentiable and even
analytic. Therefore, it is natural to demand more
of the $n$th minimal errors $e(n,s)$ and of
the information complexity $n(\e, s)$ than for those cases where
we only have finite smoothness.

To be more precise, we are interested in obtaining
exponential or uniform exponential convergence
of the minimal errors $e(n,s)$
for problems with unbounded smoothness.
We now explain how these notions are defined.
By exponential convergence
we mean that there exist functions
$q:\NN\to(0,1)$ and $p,C:\NN\to (0,\infty)$ such that
$$
e(n,s)\le C(s)\, q(s)^{\,n^{\,p(s)}}
\ \ \ \ \ \mbox{for
all} \ \ \ \ \ s, n\in \NN.
$$
Obviously, the functions $q(\cdot)$ and $p(\cdot)$ are not uniquely defined.
For instance, we can take an arbitrary number $q\in(0,1)$,
define the function $C_1$ as
$$
C_1(s)=\left(\frac{\log\,q}{\log\,q(s)}\right)^{1/p(s)},
$$
and then
$$
C(s)\, q(s)^{\,n^{\,p(s)}}=C(s)\,q^{\,(n/C_1(s))^{p(s)}}.
$$
We prefer to work with the latter bound which was also considered
in~\cite{DKPW13, KPW12}.

We say that we achieve  \emph{exponential convergence} (EXP)
for $e(n,s)$ if
there exist a number $q\in(0,1)$ and
functions $p,C,C_1:\NN\to (0,\infty)$ such that
\begin{equation}\label{exrate}
e(n,s)\le C(s)\, q^{\,(n/C_1(s))^{\,p(s)}}
\ \ \ \ \ \mbox{for
all} \ \ \ \ \ s, n\in \NN.
\end{equation}
If \eqref{exrate} holds we would like to find the largest possible
rate $p(s)$ of exponential convergence
defined as
$$
p^*(s)=\sup\{\,p(s)\ :\ \ p(s)\ \ \mbox{satisfies \eqref{exrate}}\,\}.
$$

We say that we achieve \emph{uniform exponential convergence} (UEXP)
for $e(n,s)$ if the function $p$
in \eqref{exrate} can be taken as a constant function, i.e.,
$p(s)=p>0$ for all $s\in\NN$. Similarly, let
$$
p^*=\sup\{\,p\ :\ \ p(s)=p>0\ \ \mbox{satisfies \eqref{exrate} for all
$s\in\NN$}\,\}
$$
denote the largest rate of uniform exponential convergence.
\vskip 1pc
Exponential convergence implies that asymptotically, with
respect to $\e$ tending to zero, we need $\mathcal{O}(\log^{1/p(s)}
\e^{-1})$ information evaluations to compute an
$\e$-approximation.
However, it is not
clear how long we have to wait to see this nice asymptotic
behavior especially for large $s$. This, of course, depends on
how $C(s),C_1(s)$ and $p(s)$ depend on $s$, and it is therefore
near at hand to adapt the concepts (b)--(d)
of tractability to exponential error convergence.
Indeed, we would like to replace $\e^{-1}$ by
$1+\log\,\e^{-1}$ in the standard notions (b)--(d),
which yields new versions of weak, polynomial, and
strong polynomial tractability. The following
new tractability versions (e), (f), and (g)
were already introduced in \cite{DKPW13, DLPW11, KPW12}.
We use a new kind of notation in order
to be able to distinguish (b)--(d) from (e)--(g). We say that we have:
\begin{itemize}
\item[(e)] \emph{Exponential Convergence-Weak Tractability (EC-WT)} if
$$
\lim_{s+\log\,\e^{-1}\to\infty}\frac{\log\
  n(\varepsilon,s)}{s+\log\,\e^{-1}}=0\ \ \ \ \ \mbox{with}\ \ \
 \log\,0=0\ \ \mbox{by convention}.
$$
\item[(f)] \emph{Exponential Convergence-Polynomial Tractability
    (EC-PT)} if there exist non-nega\-tive
  numbers $c,\tau_1,\tau_2$ such that
$$
n(\varepsilon,s)\le
c\,s^{\,\tau_1}\,(1+\log\,\e^{-1})^{\,\tau_2}\ \ \ \ \ \mbox{for all}\ \ \ \
s\in\NN, \ \e\in(0,1).
$$
\item[(g)]
\emph{Exponential Convergence-Strong Polynomial Tractability (EC-SPT)}
if there exist non-negative
  numbers $c$ and $\tau$  such that
$$
n(\varepsilon,s)\le
c\,(1+\log\,\e^{-1})^{\,\tau}\ \ \ \ \ \mbox{for all}\ \ \ \
s\in\NN, \ \e\in(0,1).
$$
The exponent $\tau^*$ of EC-SPT is defined as
the infimum of $\tau$ for which EC-SPT holds.
\end{itemize}

Let us give some comments on these definitions.
First, we remark that the use of the prefix EC
(exponential convergence) in (e)--(g)
is motivated by the fact that EC-PT (and therefore also
EC-SPT) implies exponential convergence (cf.~Theorem \ref{thmintectract}). 
Also EC-WT implies that $e(n,s)$
converges to zero faster than any power of $n^{-1}$ as $n$ goes to infinity,
i.e., for any $\alpha >0$
we have
\begin{equation}\label{eqlimit}
\lim_{n \rightarrow \infty} n^{\alpha} e(n,s)=0.
\end{equation}
This can be seen as follows.
Let $\alpha >0$ and choose $\delta\in (0,\frac{1}{\alpha})$.
For a fixed dimension~$s$, EC-WT implies
the existence of an $M=M(\delta) >0$ such that for
all $\varepsilon>0$ with $\log \varepsilon^{-1} >M$ we have
$$
\frac{\log n(\varepsilon,s)}{\log \varepsilon^{-1}} <
\delta \ \Leftrightarrow \ n(\varepsilon,s) < \varepsilon^{-\delta}.
$$
This implies that for large enough $n \in \NN$ we have $e(n,s)<
 n^{-1/\delta}$. Hence, we have
$n^{\alpha} e(n,s) < n^{\alpha-1/\delta} \rightarrow 0$
as $n \rightarrow \infty$.

Furthermore we note, as in \cite{DKPW13, DLPW11},
that if \eqref{exrate} holds then
\begin{equation}\label{exrate2}
n(\e,s)
\le \left\lceil C_1(s) \left(\frac{\log C(s) +
    \log \e^{-1}}{\log q^{-1}}\right)^{1/p(s)}\right\rceil
\ \ \ \ \ \mbox{for all}\ \ \ s\in \NN\ \ \mbox{and}\ \ \e\in (0,1).
\end{equation}
Moreover, if~\eqref{exrate2} holds then
$$
e(n+1,s)\le C(s)\, q^{\,(n/C_1(s))^{\,p(s)}}\ \
\ \ \ \mbox{for all}\ \ \ s,n\in \NN.
$$
This means that~\eqref{exrate} and~\eqref{exrate2} are practically
equivalent. Note that $1/p(s)$ determines the power of $\log\,\e^{-1}$
in the information complexity,
whereas $\log\,q^{-1}$ affects only the multiplier of $\log^{1/p(s)}\e^{-1}$.
From this point of view, $p(s)$ is more
important than $q$.

In particular, EC-WT means that
we rule out the cases for which
$n(\e,s)$
depends exponentially on $s$ and $\log\,\e^{-1}$.

For instance, assume that~\eqref{exrate} holds.  Then
uniform exponential convergence (UEXP) implies EC-WT if
$$
C(s)=\exp\left(\exp\left(o(s)\right)\right)
\ \ \ \mbox{and}\ \ \ C_1(s)=\exp(o(s))
\ \ \ \ \
\mbox{as}\ \ \ \ \ s\to\infty.
$$
These conditions are rather weak since
$C(s)$ can be almost doubly exponential and $C_1(s)$ almost
exponential in $s$.

The definition of EC-PT (and EC-SPT) implies
that we have uniform exponential
convergence with $C(s)=\ee$ (where $\ee$ denotes $\exp (1)$),
$q=1/\ee$, $C_1(s)=c\,s^{\,\tau_1}$ and $p=1/\tau_2$.
Obviously, EC-SPT implies
$C_1(s)=c$ and $\tau^*\le1/p^*$.

If~\eqref{exrate2} holds then we have EC-PT
if
$p:=\inf_sp(s)>0$ and
there exist non-negative numbers $A,A_1$ and $\eta,\eta_1$ such that
$$
C(s)\le \exp\left(A s^{\eta}\right)\ \ \ \mbox{and}\ \ \
C_1(s)\le A_1\,s^{\eta_1}\ \ \ \ \
\mbox{for all}\ \ \ \ \
s\in\NN.
$$
The condition on $C(s)$ seems to be quite weak since even for
singly exponential $C(s)$ we have EC-PT.
Then $\tau_1=\eta_1+\eta/p$ and $\tau_2=1/p$.
EC-SPT holds if $C(s)$ and $C_1(s)$ are uniformly bounded in $s$, and then
$\tau^*\le 1/p$.

We briefly mention a recent paper~\cite{PP13}, where a new notion of
weak tractability is defined similarly to EC-WT. Namely, let
$\kappa\ge1$. Then it is required that
\begin{equation}\label{PP13}
\lim_{s+\log\,\e^{-1}\to\infty}\frac{\log\
  n(\varepsilon,s)}{s+[\log\,\e^{-1}]^\kappa}=0\ \ \ \ \ \mbox{with}\ \ \
 \log\,0=0\ \ \mbox{by convention}.
\end{equation}
Obviously, for $\kappa=1$ this is the same as EC-WT.
However, for $\kappa>1$ the condition on WT is relaxed.
This is essential and leads to new results for linear unweighted tensor
product problems.
\vskip 1pc

In the following sections, we are going to discuss
a special choice of $H_s$ and study the problems of
$s$-variate integration and $L_2$-approximation.

\section{A weighted Korobov space of analytic functions}\label{secKor}

In this article, we choose for the Hilbert space $H_s$
a weighted Korobov space of periodic and smooth functions,
which is probably the most popular kind of space used to analyze periodic
functions. Such Korobov spaces can be defined via a
reproducing kernel (for general information on reproducing kernel
Hilbert spaces, see~\cite{Aron}) of the form
\begin{equation}\label{defkernel}
K_s(\bsx,\bsy)=\sum_{\bsh \in \ZZ^s} \rho_{\bsh}
\exp(2 \pi \icomp \bsh \cdot (\bsx-\bsy)) \ \ \ \mbox{for all}\ \ \
\bsx,\bsy\in[0,1]^s
\end{equation}
with the usual dot product
$$
\bsh\cdot(\bsx-\bsy)=\sum_{j=1}^sh_j(x_j-y_j),
$$
where
$h_j,x_j,y_j$ are the $j$th components of the vectors
$\bsh,\bsx,\bsy$, respectively. Furthermore,
$\icomp=\sqrt{-1}$. The nonnegative
$\rho_{\bsh}$ for $\bsh \in \ZZ^s$, which may also depend on $s$
and other parameters, are chosen such that
$\sum_{\bsh \in \ZZ^s} \rho_{\bsh} < \infty$. This choice
guarantees that the kernel $K_s$ is well defined, since
$$
|K_s(\bsx,\bsy)| \le  K_s(\bsx, \bsx) =
\sum_{\bsh \in \ZZ^s} \rho_{\bsh} < \infty.
$$
Obviously, the function $K_s$
is symmetric in $\bsx$ and $\bsy$ and it is easy
to show that it is also positive definite.
Therefore, $K_s(\bsx, \bsy)$ is indeed a
reproducing kernel. The corresponding Korobov space is
denoted by $H(K_s)$.

The smoothness of the functions
from $H(K_s)$ is determined by the decay of the $\rho_{\bsh}$'s.
A very well studied case in literature is for
Korobov spaces of {\it finite} smoothness~$\alpha$.
Here $\rho_{\bsh}$ is of the form
$$
\rho_{\bsh}=r_{\alpha,\bsgamma}(\bsh),
$$
where $\alpha> 1$ is a real, $\bsgamma= (\gamma_1, \gamma_2, \ldots)$
is a sequence of positive reals, and
for $\bsh = (h_1, \ldots , h_s)$ we have
$$
r_{\alpha,\bsgamma}(\bsh) =\prod_{j=1}^s r_{\alpha,\gamma_j}(h_j),
$$
with $r_{\alpha,\gamma}(0) =1$ and
$r_{\alpha,\gamma}(h)=\gamma |h|^{-\alpha}$ whenever $h \not= 0$.

Hence the $\rho_{\bsh}$'s decay polynomially in the
components of $\bsh$. The parameter $\alpha$ guarantees
the existence of some partial derivatives of the functions
and the so-called weights $\bsgamma$
model the influence of the different components
on the variation of the functions from the Korobov space.
More information can be found in \cite[Appendix~A.1]{NW08}.

The idea of introducing weights stems from Sloan and Wo\'{z}niakowski
and was first discussed in \cite{SW98}.
For multivariate integration defined over
weighted Korobov spaces of smoothness $\alpha$,
algorithms based on $n$ function evaluations can obtain
the best possible convergence rate of order
$\mathcal{O}(n^{-\alpha/2+\delta})$ for any $\delta>0$.
Under certain conditions on the weights, weak, polynomial or even strong
polynomial tractability in the sense of (b)--(d) can be achieved.
We refer to \cite{NW08,NW10,NW12} and the references therein
and to the recent survey \cite{DKS14} for further details.

Besides the case of finite smoothness, Korobov spaces of {\it infinite}
smoothness were also considered. In this case,
the $\rho_{\bsh}$'s decay to zero
exponentially fast in $\bsh$. Multivariate integration and
$L_2$-approximation for such Korobov spaces
have been analyzed in \cite{DKPW13, DLPW11, KPW12}.
To model the influence of different components
we use two weight sequences
$$
\bsa=\{a_j\}_{j \ge 1}\ \ \  \mbox{and}\ \ \
\bsb=\{b_j\}_{j \ge 1}.
$$
In order to guarantee that the kernel that we will
introduce in a moment is well defined
we must assume that $a_j>0$ and $b_j>0$.
In fact, we assume a little more throughout the paper, namely  that
with the proper ordering of variables we have
\begin{equation}\label{aabb}
0<a_1\le a_2\le \cdots\ \ \ \ \mbox{ and }\ \ \ b_\ast =\inf b_j >0.
\end{equation}
Let $a_{\ast}=\inf a_j$ which is $a_1$ in our case.

Fix $\omega\in(0,1)$ and put in \eqref{defkernel}
\begin{equation}\label{formofomega}
\rho_{\bsh}=\omega_{\bsh}:=\omega^{\sum_{j=1}^{s}a_j \abs{h_j}^{b_j}}
\qquad\mbox{for all}\qquad \bsh=(h_1,h_2,\dots,h_s)\in\ZZ^s.
\end{equation}
For this choice of $\rho_{\bsh}$
we denote the kernel in \eqref{defkernel} by $K_{s,\bsa,\bsb}$.
We suppress the dependence on $\omega$ in the notation
since $\omega$ will be fixed throughout the paper and $\bsa$ and
$\bsb$ will be varied.
Note that $K_{s,\bsa,\bsb}$ is well defined since
\begin{eqnarray*}
\sum_{\bsh \in \ZZ^s} \omega_{\bsh}
=
\prod_{j=1}^s \left(1+2 \sum_{h=1}^{\infty}
\omega^{a_j h^{b_j}}\right) \le
\left(1+2 \sum_{h=1}^{\infty} \omega^{a_\ast h^{b_\ast}}\right)^s <\infty.
\end{eqnarray*}
The last series is finite by the comparison test
because $a_\ast >0$ and $b_\ast >0$.

The Korobov space with reproducing kernel
$K_{s,\bsa,\bsb}$ is denoted by $H(K_{s,\bsa,\bsb})$.
Clearly, functions from $H(K_{s,\bsa,\bsb})$ are infinitely many
times differentiable, see \cite{DLPW11}, and they are even analytic
as shown in \cite[Proposition~2]{DKPW13}.

For $f\in H(K_{s,\bsa,\bsb})$ we have
$$
f(\bsx)=\sum_{\bsh\in\ZZ^s} \widehat f(\bsh)\,\exp(2\pi \icomp
\bsh \cdot\bsx) \ \ \ \mbox{for all}\ \ \ \bsx\in [0,1]^s,
$$
where $\widehat{f}(\bsh) =
\int_{[0,1]^s} f(\bsx) \exp(-2 \pi \icomp \bsh \cdot \bsx) \rd \bsx$
is the $\bsh$th Fourier coefficient of $f$.
The inner product of $f$ and $g$ from $H(K_{s,\bsa,\bsb})$ is given by
$$
\il f,g\ir_{H(K_{s,\bsa,\bsb})}=\sum_{\bsh\in \ZZ^s}\widehat f(\bsh)\,
\overline{\widehat g(\bsh)}\, \omega_{\bsh}^{-1}
$$ and the norm of $f$ from $H(K_{s,\bsa,\bsb})$ by
$$
\|f\|_{H(K_{s,\bsa,\bsb})}=\left(\sum_{\bsh\in \ZZ^s}|\widehat
f(\bsh)|^2\omega_{\bsh}^{-1}\right)^{1/2}<\infty.
$$
Define the functions
\begin{equation}\label{basis}
e_{\bsh}(\bsx)=\exp(2\pi\icomp\,\bsh\cdot\bsx)\,
\omega_{\bsh}^{1/2}\ \ \ \ \
\mbox{for all}\ \ \ \ \ \bsx \in[0,1]^s.
 \end{equation}
Then $\{e_{\bsh}\}_{\bsh\in\ZZ^s}$ is a complete
orthonormal basis of the Korobov space $H(K_{s,\bsa,\bsb})$.

\section{Integration in $H(K_{s,\bsa,\bsb})$}\label{secint}

In this section we study numerical integration,
i.e., we are interested in numerical approximation of
the values of integrals
\[I_s (f)=\int_{[0,1]^s}f(\bsx)\rd\bsx\ \ \ \  \
\mbox{for all}\ \ \ f\in H(K_{s,\bsa,\bsb}).\]
Using the general notation from
Section~\ref{sectractability}, we now have
$S_s(f)=I_s(f)$ for functions $f\in H_s=H(K_{s,\bsa,\bsb})$, and $G_s=\CC$.

We approximate $I_s(f)$
by means of linear algorithms $Q_{n,s}$ of the form
\[Q_{n,s}(f):=\sum_{k=1}^n q_k f(\bsx_k),\]
where coefficients $q_k\in \CC$ and sample points $\bsx_k\in[0,1)^s$.
If we choose $q_k=1/n$ for all $k=1,2,\ldots , n$
then we obtain so-called {\it quasi-Monte Carlo (QMC)
algorithms} which are often used in practical applications especially
if $s$ is large. For recent overviews of the study of
QMC algorithms we refer to \cite{DKS14,DP10,KSS11}.

The $n$th minimal worst-case error is given by
$$
e^{\mathrm{int}}(n,s)=\infp_{q_k,\bsx_k,\
  k=1,2,\dots,n}\ \sup_{f\in H(K_{s,\bsa,\bsb})
\atop \|f\|_{H(K_{s,\bsa,\bsb})}\le1\ }
\bigg|I_s(f)-\sum_{k=1}^nq_k f(\bsx_k)\bigg|.
$$
It is well known, see for instance \cite{NW10,TWW88}, that
\begin{equation}\label{wellknown}
e^{\mathrm{int}}(n,s)=\infp_{\bsx_k,\
k=1,2,\dots,n}\ \ \sup_{f\in H(K_{s,\bsa,\bsb}), \ f(\bsx_k)=0, \ k=1,2,\dots,n
\atop \|f\|_{H(K_{s,\bsa,\bsb})}\le1\ } |I_s(f)|.
\end{equation}
For $n=0$, the best we can do is to approximate $I_s(f)$ simply by zero, and
$$
e^{\mathrm{int}}(0,s)=\|I_s\|=1\ \ \ \ \ \mbox{for all}\ \ \ \ \ s\in \NN.
$$
Hence, the integration problem is well normalized for all $s$.

We now
summarize the main results regarding numerical integration
in $H(K_{s,\bsa,\bsb})$.
Here and in the following, we will be using the notational abbreviations
\begin{center}
EXP\qquad UEXP\\ WT\qquad PT\qquad SPT\\
EC-WT\qquad \ EC-PT\qquad \ EC-SPT\\
\end{center}
to denote exponential and uniform exponential convergence, and
weak, polynomial and strong polynomial tractability
in terms of (b)--(d) and (e)--(g).
We now state relations between these concepts as well as
necessary and sufficient conditions on $\bsa$ and $\bsb$ for which these
concepts hold. As we shall see, in the settings considered in this paper, many conditions for obtaining these concepts are equivalent.

We first state a theorem which describes
conditions on the weight sequences $\bsa$ and~$\bsb$
to obtain exponential (EXP)
and uniform exponential (UEXP) convergence.
This theorem is from \cite{DKPW13, KPW12}.

\begin{theorem}\label{thmint(u)exp}
Consider integration defined over the Korobov space
$H(K_{s,\bsa,\bsb})$ with weight sequences $\bsa$ and
$\bsb$ satisfying~\eqref{aabb}.
\begin{itemize}
\item\label{intexp}
EXP holds for all considered $\bsa$ and $\bsb$ and
$$
p^{*}(s)=\frac{1}{B(s)} \ \ \ \ \
\mbox{with}\ \ \ \ \ B(s):=\sum_{j=1}^s\frac1{b_j}.
$$
\item\label{intuexp}
UEXP holds iff $\bsb$ is such that
$$
B:=\sum_{j=1}^\infty\frac1{b_j}<\infty.
$$
If so then $p^*=1/B$.
\end{itemize}
\end{theorem}

Theorem~\ref{thmint(u)exp} states that
we always have exponential convergence.
However, a necessary and sufficient condition
for uniform exponential convergence
is that the weights~$b_j$ go to infinity so fast
that $B:=\sum_{j=1}^\infty b_j^{-1}<\infty$, with
no extra conditions on $a_j$
and $\omega$. The largest exponent $p$ of uniform exponential
convergence is $1/B$. Hence for small $B$ the exponent $p$ is large.
For instance, for $b_j=j^{-2}$ we have $B=\pi^2/6$ and
$p^*=6/\pi^2=0.6079\dots$.

Next, we consider standard notions of tractability, (b)--(d).
They have not yet been studied for the Korobov space
$H(K_{s,\bsa,\bsb})$ and therefore we need to prove the next
theorem.

\begin{theorem}\label{PTint}
Consider integration defined over the Korobov space
$H(K_{s,\bsa,\bsb})$ with weight sequences $\bsa$ and $\bsb$
satisfying~\eqref{aabb}.
For simplicity, assume that
$$
A:=\lim_{j \rightarrow \infty} \frac{a_j}{\log j}
$$
exists.
\begin{itemize}
\item\label{SPT} SPT holds if
$A >\frac{1}{\log \omega^{-1}}$.
In this case the exponent $\tau^\ast$ of SPT satisfies
$$
\tau^\ast \le  \min\left(2, \frac{2}{A \log \omega^{-1}}\left(
1+\frac1{A \log \omega^{-1}}\right)\right).
$$
On the other hand, if we have SPT with exponent $\tau^\ast$, 
then $A \ge \frac{1}{\tau^{\ast} \log \omega^{-1}}$.
\item\label{PT}
PT holds if there is an integer $j_0\ge2$ such that
$$
 \frac{a_j}{\log\,j}\ge\frac1{\log\,\omega^{-1}}\ \ \ \ \mbox{for all}\
   \ \ \ \ j\ge j_0.
$$
\item WT holds if $\lim_{j \rightarrow \infty} a_j=\infty$.
\end{itemize}
\end{theorem}
\begin{proof}
It is well known that integration is no harder than
$L_2$-approximation for the class $\lstd$. For the Korobov
class the initial errors of integration and approximation are $1$.
Therefore the corresponding notions of tractability for approximation
imply the same notions of tractability for integration.
From Theorem~\ref{PTapprox}, presented in the next section, we thus conclude SPT, PT and WT also for
integration. The second bound on the exponent $\tau^*$ of SPT also
follows from Theorem~\ref{PTapprox}.
It remains to prove that
$\tau^*\le 2$. It is known, see, e.g., \cite[Theorem 10.4]{NW10},
that
$$
[e^{\rm int}(n,s)]^2\le
\frac1{n}\,\int_{[0,1]^s}K_{s,\bsa,\bsb}(\bsx,\bsx)\,{\rm
  d}\bsx\le\frac1n\,\prod_{j=1}^s\left(1+
2\sum_{h=1}^\infty
\omega^{a_jh^{b^*}}\right).
$$
It is shown in the proof of Theorem~\ref{PTapprox} (with $\tau=1$)
that $A>0$ implies the existence of $C\in(0,\infty)$ such that
$2\sum_{h=1}^\infty\omega^{a_jh^{b^*}}\le C\,\omega^{a_j}$.
Therefore
$$
[e^{\rm int}(n,s)]^2\le
\frac1n\prod_{j=1}^s\left(1+C\sum_{j=1}^\infty\omega^{a_j}\right)\le
 \frac1n\,
\exp\left(C\sum_{j=1}^s\omega^{a_j}\right).
$$
Note that for $j\ge2$ we have
$\omega^{a_j}=j^{-a_j\,(\log \omega^{-1})/\log j}$.
Since $A>1/(\log\,\omega^{-1})$ for large~$j$ we conclude that
$\omega^{a_j}\le j^{-\beta}$ with $\beta\in(1,A \log\,\omega^{-1})$.
Hence $\sum_{j=1}^\infty \omega^{a_j}<\infty$ and
$e(n,s)\le \e$ for $n=\mathcal{O}(\e^{-2})$ with the factor in the big
$\mathcal{O}$ notation independent of~$s$. This implies SPT with the
exponent at most $2$. 

It remains to show the necessary condition for SPT with exponent $\tau^\ast$.
First of all we show the estimate 
\begin{equation}\label{lowest}
e^{\mathrm{int}}(s,s)\ge\frac{\omega^{a_s}}{\sqrt{1+\omega^{2a_s}}} \ \ \ \ \
\mbox{for all}\ \ \ \ \ s\in\NN.
\end{equation}
Let $\bsh^{(0)}=(0,0,\dots,0)\in\ZZ^s$. For $j=1,2,\dots,s$,  let
$$
\bsh^{(j)}=(0,0,\dots,1,0,\dots,0)\in\ZZ^s\ \ \
\mbox{with $1$ on the $j$th place}.
$$
For $\bsh\in\ZZ^s$, let
$$
c_{\bsh}(\bsx)=\exp\left(2\pi\,{\rm i}\,\sum_{j=1}^sh_jx_j\right)\ \ \ \ \
\mbox{for all}\ \ \ \ \ \bsx\in[0,1]^s.
$$
For $j=0,1,\dots,s$, note that
$c_{\bsh^{(j)}}(\bsx)=\exp(2\,\pi\,{\rm i}\,
x_j)$ and
$$
c_{\bsh^{(j)}}(\bsx)\,\overline{c_{\bsh^{(k)}}(\bsx)}=
c_{\bsh^{(j)}-\bsh^{(k)}}(\bsx).
$$

Consider the function
$$
f(\bsx)=\sum_{j=0}^s \alpha_j\,c_{\bsh^{(j)}}(\bsx)\ \ \ \ \
\mbox{for all}\ \ \ \ \ \bsx\in[0,1]^s
$$
for some complex numbers $\alpha_j$.

We know that adaption does not help for the integration problem.
Suppose that we sample functions at $s$ nonadaptive points
$\bsx_1,\bsx_2,\dots,\bsx_s\in[0,1]^s$. We choose numbers $\alpha_j$ such that
$$
f(\bsx_j)=0\ \ \ \ \ \mbox{for all}\ \ \ \ \ j=1,2,\dots,s.
$$
This corresponds to $s$ homogeneous linear equations in $(s+1)$ unknowns.
Therefore there exists a nonzero solution $\alpha_0,\alpha_1,\dots,\alpha_s$
which we may normalize such that
$$
\sum_{j=0}^s|\alpha_j|^2=1.
$$
Let
$$
g(\bsx)=\overline{f(\bsx)}\,f(\bsx)=\sum_{j,k=0}^s\alpha_j\,
\overline{\alpha}_k\,c_{\bsh^{(j)}-\bsh^{(k)}}(\bsx)\ \ \ \ \
\mbox{for all}\ \ \ \ \ \bsx\in[0,1]^s.
$$
Clearly,
$g(\bsx_j)=0$ for all $j=1,2,\dots,s$. Since $I_s(c_{\bsh^{(j)}-\bsh^{(k)}})=0$
for $j\not=k$ and $1$ for $j=k$ we obtain
$$
I_s(g)=\sum_{j=0}^s|\alpha_j|^2=1.
$$
Now it follows from \eqref{wellknown} that
$$
e^{\mathrm{int}}(s,s)\ge 
I_s\left(\frac{g}{\|g\|_{H(K_{s,\bsa,\bsb})}}\right)=
\frac1{\|g\|_{H(K_{s,\bsa,\bsb})}}.
$$
This is why we need to estimate the norm of $g$ from above.
Note that
\begin{eqnarray*}
\|g\|_{H(K_{s,\bsa,\bsb})}^2&=&
\il g,g\ir_{H(K_{s,\bsa,\bsb})}\\&=&\il
\sum_{j_1,k_1=0}^s\alpha_{j_1}\,\overline{\alpha_{k_1}}\,
c_{\bsh^{(j_1)}-\bsh^{(k_1)}},
\sum_{j_2,k_2=0}^s\alpha_{j_2}\,\overline{\alpha_{k_2}}\,
c_{\bsh^{(j_2)}-\bsh^{(k_2)}}\ir_{H(K_{s,\bsa,\bsb})} \\
&=&
\sum_{j_1,k_1,j_2,k_2=0}^s\alpha_{j_1}\,\overline{\alpha_{j_2}}\,
\overline{\alpha_{k_1}}\,\alpha_{k_2}\,
\il c_{\bsh^{(j_1)}-\bsh^{(k_1)}},
c_{\bsh^{(j_2)}-\bsh^{(k_2)}}\ir_{H(K_{s,\bsa,\bsb})}.
\end{eqnarray*}
For $\bsh^{(j_1)}-\bsh^{(k_1)}\not=\bsh^{(j_2)}-\bsh^{(k_2)}$ we have
$$
\il c_{\bsh^{(j_1)}-\bsh^{(k_1)}},
c_{\bsh^{(j_2)}-\bsh^{(k_2)}}\ir_{H(K_{s,\bsa,\bsb})}=0,
$$
whereas for $\bsh^{(j_1)}-\bsh^{(k_1)}=\bsh^{(j_2)}-\bsh^{(k_2)}$ we have
$$
\il c_{\bsh^{(j_1)}-\bsh^{(k_1)}},
c_{\bsh^{(j_2)}-\bsh^{(k_2)}}\ir_{H(K_{s,\bsa,\bsb})}=
\omega^{-1}_{\bsh^{(j_1)}-\bsh^{(k_1)}}.
$$
Therefore it is enough to consider
$$
\bsh^{(j_1)}-\bsh^{(k_1)}=\bsh^{(j_2)}-\bsh^{(k_2)}.
$$
Suppose first that $j_1\not=k_1$. Then
$\bsh^{(j_1)}-\bsh^{(k_1)}=\bsh^{(j_2)}-\bsh^{(k_2)}$ implies that
$j_2=j_1$ and $k_2=k_1$ and
$$
\omega^{-1}_{\bsh^{(j_1)}-\bsh^{(k_1)}}=\omega^{-a_{j_1}-a_{k_1}}.
$$
On the other hand, if $j_1=k_1$ then
$\bsh^{(j_1)}-\bsh^{(k_1)}=\bsh^{(0)}$ which implies that $j_2=k_2$ and
$$
\omega^{-1}_{\bsh^{(j_1)}-\bsh^{(k_1)}}=1.
$$
Therefore
\begin{eqnarray*}
\|g\|_{H(K_{s,\bsa,\bsb})}^2&=&
\sum_{j_1,k_1,j_2,k_2=0, \ \bsh^{(j_1)}-\bsh^{(k_1)}=
\bsh^{(j_2)}-\bsh^{(k_2)}}^s\alpha_{j_1}\,
\overline{\alpha_{j_2}}\,\overline{\alpha_{k_1}}\,
\alpha_{k_2}\,\omega_{\bsh^{(j_1)}-\bsh^{(k_1)}}\\
&=&
\sum_{j_1=0}^s\,\sum_{k_1=0,k_1\not=j_1}^s|\alpha_{j_1}|^2\,|\alpha_{k_1}|^2\,
\omega^{-a_{j_1}-a_{k_1}}\, +\, 
\sum_{j_1=0}^s|\alpha_{j_1}|^2\,\sum_{j_2=0}^s|\alpha_{j_2}|^2\\
&=&\sum_{j=0}^s|\alpha_j|^2\omega^{-a_j}\,\left(-|\alpha_j|^2\omega^{-a_j}\,+\,
\sum_{k=0}^s|\alpha_k|^2\omega^{-a_k}\right)+1\\
&\le& \left(\sum_{j=0}^s|\alpha_j|^2\omega^{-a_j}\right)^2+1
\le \omega^{-2a_s}+1.
\end{eqnarray*}
Hence,
$$
\|g\|_{H(K_{s,\bsa,\bsb})}\le
\sqrt{1+\omega^{-2a_s}}=\frac{\sqrt{1+\omega^{2a_s}}}{\omega^{a_s}}.
$$
Finally,
$$
e^{\mathrm{int}}(s,s)\ge 
\frac1{\|g\|_{H(K_{s,\bsa,\bsb})}}\ge 
\frac{\omega^{a_s}}{\sqrt{1+\omega^{2a_s}}},
$$
and thus \eqref{lowest} is shown.

Assume that we have SPT with the exponent $\tau^*$. This means that
for any positive $\delta$ there exists a positive number $C_\delta$ such that
$$
n(\e,s)\le C_\delta\,\e^{-(\tau^*+\delta)}\ \ \ \ \
\mbox{for all}\ \ \ \ \ \e\in(0,1),\ s\in\NN.
$$
Let $n=n(\e):=\lfloor C_{\delta}\,\e^{-(\tau^*+\delta)}\rfloor$. 
Then
$$
e^{\mathrm{int}}(n(\e),s)\le \e
\ \ \ \ \ \mbox{for all}\ \ \ \ \ s\in\NN.
$$
Taking 
$s=n(\e)$, we conclude from~\eqref{lowest} that 
$$
\frac{\omega^{a_s}}{\sqrt{1+\omega^{2a_s}}}
\le e^{\mathrm{int}}(s,s)\le \e,
$$
which implies
$$
(1-\e^2)\omega^{2a_s} \le \e^2.
$$
Taking logarithms this means that
$$
\frac{a_s}{\log\,\e^{-1}}\ge \frac{1+o(1)}{\log\,\omega^{-1}}\ \ \ \ \ 
\mbox{as}\ \ \ \ \ \e\to0.
$$
Since $\log\,\e^{-1}=(1+o(1))(\tau^*+\delta)^{-1}\,\log\,s$ we finally
have 
$$
A=\lim_{s\to \infty}\frac{a_s}{\log\,s}\ge \frac1{(\tau^*+\delta)\,\log\,
\omega^{-1}}.
$$
Since $\delta$ can be arbitrarily small, the proof is completed.
\end{proof}

We stress that for integration we only know sufficient conditions on $\bsa$ and
$\bsb$ for the standard
notions PT and WT. Obviously, it would be
welcome to find also necessary conditions and verify if they match
the conditions presented in the last theorem. 
For SPT we have a sufficient condition and a necessary condition, but
there remains a (small) gap between these. 
Again, it would be welcome to find matching sufficient and 
necessary conditions for SPT.
Note that it may happen that $A=\infty$. This happens when $a_j$'s go
to infinity faster than $\log\,j$. In this case, the exponent
of SPT is zero. This means that for any positive $\delta$, no matter
how small, $n(\e,s)=\mathcal{O}(\e^{-\delta})$ with the factor in the
big $\mathcal{O}$ notation independent of~$s$.
We also stress that the conditions on
all standard notions of tractability depend only on $\bsa$ and are
independent of $\bsb$.

Finally, we have a result regarding
the EC notions of tractability, (d)--(f).
The subsequent theorem follows by combining
the findings in \cite{KPW12} and \cite[Section~9]{DKPW13}.

\begin{theorem}\label{thmintectract}
Consider integration defined over the Korobov space
$H(K_{s,\bsa,\bsb})$ with weight sequences $\bsa$ and
$\bsb$ satisfying~\eqref{aabb}.
Then the following results hold:
\begin{itemize}
\item\label{intecpt}
EC-PT (and, of course, EC-SPT) implies UEXP.
\item\label{intecwt}
We have
\begin{eqnarray*}
\mbox{EC-WT}\ &\Leftrightarrow&\ \lim_{j\to\infty}a_j=\infty,\\
\mbox{EC-WT+UEXP}\ &\Leftrightarrow&\  B<\infty\ \ \mbox{and}\ \
\lim_{j\to\infty}a_j=\infty.
\end{eqnarray*}
\item\label{intecequiv}
The following notions are equivalent:
\begin{multline*}
\ \ \ \ \mbox{EC-PT}\ \Leftrightarrow\ \mbox{EC-PT+EXP}\ \Leftrightarrow\
\mbox{EC-PT+UEXP} \\ \Leftrightarrow \mbox{EC-SPT}\ \Leftrightarrow\
\mbox{EC-SPT+EXP}\ \Leftrightarrow\ \mbox{EC-SPT+UEXP}.
\end{multline*}
\item \label{intecspt}
EC-SPT+UEXP holds iff $b_j^{-1}$'s are summable and $a_j$'s are
exponentially large in~$j$, i.e.,
$$
B:=\sum_{j=1}^\infty\frac1{b_j}<\infty\quad\mbox{and}\quad
\alpha^*:=\liminf_{j\to\infty}\frac{\log\,a_j}j>0.
$$
Then the exponent $\tau^*$ of EC-SPT satisfies
$$
\tau^*\in \left[B,B+\min\left(B,\frac{\log\,3}{\alpha^*}\right)\right].
$$
In particular, if $\alpha^* =\infty$ then $\tau^* = B$.
\end{itemize}
\end{theorem}

Theorem~\ref{thmintectract} states that EC-PT implies UEXP
and hence $B<\infty$. The notion of EC-PT is
therefore stronger than the notion of uniform exponential
convergence.
EC-WT holds if and only if the $a_j$'s tend to infinity.
This holds independently of the weights~$\bsb$
and independently of the rate of convergence of $\bsa$ to infinity.
As already shown, this implies that \eqref{eqlimit} holds.
Furthermore,
EC-WT+UEXP holds if additionally $B< \infty$.
Hence for $\lim_j a_j=\infty$ and $B=\infty$, EC-WT holds without
UEXP. It is a bit surprising
that the notions of EC-tractability with uniform exponential
convergence are equivalent. Necessary and sufficient
conditions for EC-SPT with uniform
exponential convergence are $B< \infty$ and $\alpha^{\ast}>0$.
The last condition means that $a_j$'s are exponentially large
in $j$ for large $j$.

\medskip

\section{$L_2$-approximation in $H(K_{s,\bsa,\bsb})$}\label{secapp}

Let us now turn to approximation in the space $H(K_{s,\bsa,\bsb})$.
We study $L_2$-approximation of functions from
$H(K_{s,\bsa,\bsb})$. This problem is defined as an approximation
of the embedding from the space $H(K_{s,\bsa,\bsb})$ to the space
$L_2([0,1]^s)$, i.e.,
$$
{\rm EMB}_s:H(K_{s,\bsa,\bsb}) \rightarrow L_2([0,1]^s)\ \ \
\mbox{given by}\ \ \ {\rm EMB}_s(f)=f.
$$
In terms of the notation in Section~\ref{sectractability},
$S_s(f)={\rm EMB}_s(f)=f$ for $f\in H(K_{s,\bsa,\bsb})$,
and $G_s=L_2 ([0,1]^s)$.

Without loss of generality, see
again \cite{NW08,TWW88},
we approximate ${\rm EMB}_s$ by  linear algorithms~$A_{n,s}$ of the form
\begin{equation}\label{linalg}
  A_{n,s}(f) = \sum_{k=1}^{n}\alpha_k L_k(f)\ \ \ \
\mbox{for} \ \ \ \ \ f \in H(K_{s,\bsa,\bsb}),
\end{equation}
where each $\alpha_k$ is a function from $L_{2}([0,1]^{s})$ and
each $L_k$ is a continuous linear functional defined on $H_s$
from a permissible class $\Lambda$ of information,
$\Lambda\in\{\lall,\lstd\}$.
Since $H(K_{s,\bsa,\bsb})$ is a reproducing kernel Hilbert space,
function evaluations are continuous linear functionals and therefore
$\Lambda^{\mathrm{std}}\subseteq \Lambda^{\mathrm{all}}$.

Let $e^{L_2-\mathrm{app},\Lambda}(n,s)$ be the $n$th minimal
worst-case error,
$$
e^{L_2-\mathrm{app},\Lambda}(n,s) = \inf_{A_{n,s}}
e^{L_2-\mathrm{app}}(H(K_{s,\bsa,\bsb}),A_{n,s}),
$$
where the infimum is taken
over all linear algorithms $A_{n,s}$ of the form \eqref{linalg} 
using information
from the class $\Lambda\in\{\Lambda^{\mathrm{all}},\Lambda^{\mathrm{std}}\}$.
For $n=0$ we simply approximate $f$ by zero, and
the initial error is
$$
e^{L_2-\mathrm{app},\Lambda}(0,s) = \|{\rm EMB}_s\|=
\sup_{f \in H(K_{s,\bsa,\bsb}) \atop \norm{f}_{H(K_{s,\bsa,\bsb})}\le
  1}
\norm{f}_{L_{2}([0,1]^s)} = 1.
$$
This means that also $L_2$-approximation is well normalized for all
$s\in\NN$.

Let us now outline the main results
regarding $L_2$-approximation in $H(K_{s,\bsa,\bsb})$.
Again, we start with results on EXP and UEXP. The following result was
proved in~\cite{DKPW13}.

\begin{theorem}\label{thmapp(u)exp}
Consider $L_2$-approximation defined over the Korobov space
$H(K_{s,\bsa,\bsb})$ with weight sequences $\bsa$ and $\bsb$
satisfying~\eqref{aabb}.
Then the following results hold for both classes
$\Lambda^{\rm{all}}$ and $\Lambda^{\rm{std}}$:
\begin{itemize}
\item\label{appexp}
EXP holds for all considered $\bsa$ and $\bsb$ with
$$
p^{*}(s)=\frac{1}{B(s)} \ \ \ \ \
\mbox{with}\ \ \ \ \ B(s):=\sum_{j=1}^s\frac1{b_j}.
$$
\item\label{appuexp}
UEXP holds iff $\bsa$ is an arbitrary sequence and
$\bsb$ is such that
$$
B:=\sum_{j=1}^\infty\frac1{b_j}<\infty.
$$
If so then $p^*=1/B$.
\end{itemize}
\end{theorem}
Note that the conditions are the same as for
the integration problem in Theorem~\ref{thmint(u)exp}.
Hence the comments following Theorem~\ref{thmint(u)exp} also apply
for approximation. Beyond that it is interesting that
we have the same conditions for $\Lambda^{\rm{all}}$
and $\Lambda^{\rm{std}}$, although the class
$\Lambda^{\rm{std}}$ is much smaller than the class $\Lambda^{\rm{all}}$.

We now address conditions on the weights $\bsa$ and $\bsb$
for the standard concepts of tractability. This has not yet been done
before for $\omega_{\bsh}$ of the form~\eqref{formofomega},
and therefore we need to prove the next theorem.
\begin{theorem}\label{PTapprox}
Consider $L_2$-approximation defined over
the Korobov space $H(K_{s,\bsa,\bsb})$ with arbitrary sequences $\bsa$ and
$\bsb$ satisfying~\eqref{aabb}.
Assume for simplicity that
$$
A:=\lim_{j\to\infty}\frac{a_j}{\log\,j}
$$
exists. Then the following results hold:

For $\Lambda^{\rm{all}}$ we have:
\begin{itemize}
\item
$
\mbox{SPT}\ \ \ \Leftrightarrow\ \ \ A>0.$

In this case, the exponent of SPT is
$$
[\tau^{\rm all}]^*=
\frac{2}{A\,\log\,\omega^{-1}}.
$$
\item
$
\mbox{PT}\ \ \  \Leftrightarrow\ \ \ \mbox{SPT}.
$
\item \mbox{WT}  holds for all considered
$\bsa$ and $\bsb$.
\end{itemize}

For $\Lambda^{\rm{std}}$ we have:
\begin{itemize}
\item SPT holds if $A>1/(\log \omega^{-1})$.
In this case, the exponent $[\tau^{{\rm std}}]^\ast$
satisfies
$$
[\tau^{\rm all}]^*\le
[\tau^{{\rm std}}]^\ast \le  [\tau^{\rm{all}}]^\ast +
\frac{1}{2}([\tau^{\rm{all}}]^\ast)^2 < [\tau^{\rm{all}}]^\ast +2.$$
On the other hand, if we have SPT with exponent $[\tau^{\rm{std}}]^\ast$, 
then $A \ge \frac{1}{[\tau^{\rm{std}}]^\ast \log \omega^{-1}}$.
\item PT holds if
there is an integer $j_0\ge2$ such that
$$
 \frac{a_j}{\log\,j}\ge\frac1{\log\,\omega^{-1}}\ \ \ \ \mbox{for all}\
   \ \ \ \ j\ge j_0.
$$
\item
WT holds if $\lim_{j \rightarrow \infty} a_j=\infty$.
\end{itemize}
\end{theorem}
\begin{proof}
Consider first the class $\lall$.
\begin{itemize}
\item From \cite[Theorem~5.2]{NW08} it follows that SPT
for $\Lambda^{\rm{all}}$ is equivalent to the existence of a number
$\tau>0$ such that
$$
C_{{\rm SPT},\tau}:= \sup_s \left( \sum_{\bsh \in \ZZ^s}
  \omega_{\bsh}^{\tau}\right)^{1/\tau}   < \infty.
$$
Note that
$$
\sum_{\bsh \in \ZZ^s} \omega_{\bsh}^{\tau} =
\prod_{j=1}^s \left(1+2 \sum_{h=1}^\infty \omega^{\tau a_j
    h^{b_j}}\right)=
\prod_{j=1}^s \left(1+2\omega^{\tau a_j}\,\sum_{h=1}^\infty \omega^{\tau a_j
    (h^{b_j}-1)}\right).
$$
We have
$$
1\le\sum_{h=1}^\infty \omega^{\tau a_j(h^{b_j}-1)}
\le \sum_{h=1}^\infty \omega^{\tau a_*(h^{b_*}-1)}
=:A_\tau.
$$
We can rewrite $A_\tau$ as
$$
A_\tau=\sum_{h=1}^\infty h^{-x_h},
$$
where $x_1=1$ and for $h \ge 2$ we have
$$
x_h=\tau\,a_* (\log\,\omega^{-1})\,\frac{h^{b_*}-1}{\log\, h}.
$$
Since $\lim_hx_h=\infty$ the last series is convergent and therefore
$A_\tau<\infty$. This proves that
$$
\sum_{\bsh \in \ZZ^s}
\omega_{\bsh}^{\tau}=\prod_{j=1}^s\left
(1+2A(\tau)\,\omega^{\tau a_j}\right)\ \ \ \mbox{with}\ \ \ \
A(\tau)\in\left[1,A_\tau\right].
$$
This implies that
$$
\sup_s\,\left(\sum_{\bsh \in \ZZ^s}
\omega_{\bsh}^{\tau}\right)^{1/\tau}=
\prod_{j=1}^\infty
\left(1+2A(\tau)\,
\omega^{\tau a_j}\right)^{1/\tau}<\infty\ \ \ \ \mbox{iff}\ \ \ \
\sum_{j=1}^\infty
\omega^{\tau a_j}<\infty.
$$
We now show that
$$
\sum_{j=1}^\infty
\omega^{\tau a_j}<\infty\ \ \mbox{for some}\ \ \tau \ \
\ \ \mbox{iff}\ \ \ \ A >0.
$$
Indeed, for $j\ge2$ we can write $\omega^{\tau a_j}=j^{-y_j}$ with
$$
y_j=\tau \log\,\omega^{-1}\ \frac {a_j}{\log\,j}.
$$
If $A>0$ then for an arbitrary positive $\delta$
we can choose  $\tau$ such that $y_j\ge 1+\delta$ for sufficiently
large $j$ and therefore the series
$$
\sum_{j=1}^\infty\omega^{\tau a_j}=\omega^{\tau a_1}+\sum_{j=2}^\infty
j^{-y_j}
$$
is convergent.

If $A=0$ then independently of $\tau$ the series
$\sum_{j=1}^\infty \omega^{\tau a_j}$ is
divergent. Indeed,  then $\lim_jy_j=0$ and
for an arbitrary positive $\delta\le 1$ and $\tau$ we
can choose $j(\delta,\tau)$ such that $y_j\in(0,\delta)$ for all
$j\ge j(\delta,\tau)$ and
$$
\sum_{j=1}^\infty\omega^{\tau a_j}\ge
\sum_{j=j(\delta,\tau)}^\infty j^{-\delta}=\infty,
$$
as claimed. This proves that SPT holds iff $A>0$.

Furthermore, \cite[Theorem~5.2]{NW08} states that the exponent
of SPT is $2\tau^*$, where $\tau^*$ is the infimum of $\tau$ for which
$C_{{\rm SPT},\tau}<\infty$. In our case, it is clear that we must have
$\tau\ge (1+\delta)/((A-\delta)\log\,\omega^{-1})$ for arbitrary
$\delta\in (0,A)$. This completes the proof of this point.
\item To show that PT is equivalent to SPT,
it is obviously enough to show that PT
implies SPT. According to \cite[Theorem~5.2]{NW08},
PT for $\Lambda^{\rm{all}}$ is equivalent to the existence of numbers
$\tau>0$ and $q \ge 0$ such that
$$
C_{{\rm PT}}:= \sup_s \left(
\sum_{\bsh \in \ZZ^s}
\omega_{\bsh}^{\tau}\right)^{1/\tau} s^{-q} < \infty.
$$
This means that
\begin{equation}\label{eqlogCPT}
\log\, \sum_{\bsh \in \ZZ^s} \omega_{\bsh}^{\tau}
\le \tau\left(\log\,C_{\rm PT}\ +\ q\,\log\,s\right).
\end{equation}
From the previous considerations we know that
$$
\log\, \sum_{\bsh \in \ZZ^s}
\omega_{\bsh}^{\tau}=\log\,\prod_{j=1}^s\left(1+2A(\tau)\omega^{\tau
    a_j}\right)=\sum_{j=1}^s\log\, (1+2A(\tau)\omega^{\tau a_j}).
$$
Assume that $A=0$. Suppose first that $a_j$'s are uniformly bounded.
Then\linebreak $\log\, \sum_{\bsh \in \ZZ^s} \omega_{\bsh}^{\tau}$
is of order $s$ which contradicts the inequality \eqref{eqlogCPT}.
Assume now that $\lim_ja_j=\infty$. Then
$\log\, \sum_{\bsh \in \ZZ^s} \omega_{\bsh}^{\tau}$
is of order $\sum_{j=2}^s\omega^{\tau a_j}=\sum_{j=2}^sj^{-y_j}$.
Since $\lim_jy_j=0$ we have for $\delta\in(0,1)$, as before,
$j^{-y_j}\ge j^{-\delta}$ for large $j$. This proves that
$\sum_{j=2}^sj^{-\delta}\approx \int_2^sx^{-\delta}\,{\rm d}x$
is of order $s^{1-\delta}$ which again
contradicts the inequality \eqref{eqlogCPT}.
Hence, $A>0$ and we have SPT.
\item
We now show WT for all $\bsa$ and $\bsb$ with $a_*,b_*>0$.
We have
$$
\omega_{\bsh}=\omega^{\sum_{j=1}^sa_j|h_j|^{b_j}}\le
\omega_{*,\bsh}:=\omega^{a_*\,\sum_{j=1}^s|h_j|^{b_*}}.
$$
Note that for $\bsh={\bf 0}$ we have
$\omega_{\bsh}=\omega_{*,\bsh}=1$. This shows that the approximation
problem with $\omega_{\bsh}$ is not harder than the approximation
problem with $\omega_{*,\bsh}$. The latter problem is a linear tensor
product problem with the univariate eigenvalues of
$W_1={\rm EMB}_1^*\,{\rm EMB}_1: H(K_{1,a_*,b_*})\to H(K_{1,a_*,b_*})$ given by
$$
\lambda_1=1, \ \ \ \ \lambda_{2j}=\lambda_{2j+1}=\omega^{a_*\,j^{b_*}}
\ \ \ \ \mbox{for all}\ \ \ \ j\ge1.
$$
Clearly, $\lambda_2<\lambda_1$ and
$\lambda_j$ goes to zero faster than polynomially with $j$.
This implies WT due to
\cite[Theorem~5.5]{NW08}.\footnote{
In fact, we also have quasi-polynomial tractability,
i.e., $n(\e,s)\le C\,\exp(t(1+\log\,\e^{-1})(1+\log\,s))$
for some $C>0$ and $t\approx 2/(a_*\log\,\omega^{-1})$,
see \cite{GW11}.}
\end{itemize}
We now turn to the class $\lstd$.
\begin{itemize}
\item  For $A>1/(\log\,\omega^{-1})$ we have $[\tau^{\rm all}]^*<2$. {}From
\cite[Theorem 26.20]{NW12} we get SPT for $\lstd$ as well as the
bounds on $[\tau^{\rm std}]^*$. The necessary condition for SPT with exponent $[\tau^{\rm std}]^*$ follows from Theorem~\ref{PTint}.
\item To obtain PT we use \cite[Theorem 26.13]{NW12} which states that
polynomial tractabilities for $\lstd$ and $\lall$ are equivalent if
${\rm trace}(W_s)=\mathcal{O}(s^{\,q})$ for some $q\ge0$,
where ${\rm trace}(W_s)$ is the sum of the eigenvalues of the operator
$$W_s={\rm EMB}_s^*\,{\rm EMB}_s: \ H(K_{s,\bsa,\bsb})\to
H(K_{s,\bsa,\bsb}).
$$
In our case, $W_s$ is given by
$$
W_sf=\sum_{\bsh\in\ZZ^s}\omega_{\bsh}\il
f,e_{\bsh}\ir_{H(K_{s,\bsa,\bsb})}e_{\bsh}
$$
with $e_{\bsh}$ given by~\eqref{basis}.
The eigenpairs of $W_s$ are $(\omega_{\bsh},e_{\bsh})$ since
$$
W_se_{\bsh}=\omega_{\bsh}e_{\bsh}=
\omega^{\,\sum_{j=1}^sa_j|h_j|^{b_j}}\, e_{\bsh}\ \ \ \ \
\mbox{for all}\ \ \ \ \ \bsh\in \ZZ^s
$$
and hence
$$
{\rm trace}(W_s)
=\prod_{j=1}^s\left(1+2A(1)\omega^{a_j}\right)\le
\exp\left(2A(1)\sum_{j=1}^s\omega^{a_j}\right).
$$
Due to the assumption $a_j/\log\,j\ge 1/(\log\,\omega^{-1})$ for $j\ge
j_0$ we have $\omega^{a_j}\le j^{-1}$ for $j\ge j_0$. Therefore
there is a positive $C$ such that
$$
{\rm trace}(W_s)\le C\exp\left(A(1)\sum_{j=j_0}^sj^{-1}\right)\le C\,s^{A(1)}.
$$
This proves that PT for $\lstd$ holds iff PT for $\lall$ holds.
As we already proved, the latter holds iff $A>0$.
The assumption on $a_j$ implies that $A\ge1/(\log\,\omega^{-1})>0$.
\item To obtain WT we use \cite[Theorem 26.11]{NW12}.
This theorem states that weak tractabilities
for classes $\lstd$ and $\lall$ are equivalent if
$\log\,{\rm trace}(W_s)=o(s)$.
The proof of Theorem~\ref{thmapp(u)exp} in \cite{DKPW13} yields that
$\lim_ja_j=\infty$ implies $\sum_{j=1}^s\omega^{a_j}=o(s)$.
Hence,
$$
\log\,{\rm trace}(W_s)\le \log\left(\exp\left(2A(1)\,o(s)\right)\right)=o(s),
$$
as needed.
\end{itemize}
\end{proof}
We briefly comment on Theorem~\ref{PTapprox}.
For the class $\lall$ we know necessary and sufficient conditions on
SPT, PT and WT if the limit of $a_j/\log\, j$ exists.
It is interesting to study the case when the last limit does not
exist. It is easy to check that $A_{\rm inf}:=\liminf_ja_j/\log\,j>0$
implies SPT but it is not clear whether SPT implies $A_{\rm inf}>0$.

For the class $\lstd$ we only know sufficient
conditions for PT and WT. It would be of interest to verify if these conditions
are also necessary. 
For SPT, as for multivariate integration,
there remains a (small) gap between 
sufficient and necessary conditions. Again it would be desirable 
to close this gap.

Finally, we have results regarding the EC-notions of tractability, (e)--(g).
The subsequent theorem has been shown in \cite{DKPW13}.

\begin{theorem}\label{thmappectract}
Consider $L_2$-approximation defined over the Korobov space
$H(K_{s,\bsa,\bsb})$ with arbitrary sequences $\bsa$ and
$\bsb$ satisfying~\eqref{aabb}.
Then the following results hold for both classes
$\Lambda^{\rm{all}}$ and $\Lambda^{\rm{std}}$:
\begin{itemize}
\item\label{appecpt}
EC-PT (and, of course, EC-SPT)
tractability implies uniform exponential convergence,
$
\mbox{EC-PT}\  \ \Rightarrow\ \  \mbox{UEXP}.
$
\item\label{appecwt}
We have
\begin{eqnarray*}
\mbox{EC-WT}\ &\Leftrightarrow&\ \lim_{j\to\infty}a_j=\infty,\\
\mbox{EC-WT+UEXP}\ &\Leftrightarrow&\  B<\infty\ \ \mbox{and}\ \
\lim_{j\to\infty}a_j=\infty.
\end{eqnarray*}
\item\label{appecequiv}
The following notions are equivalent:
\begin{eqnarray*}
\ \ \ \ \mbox{EC-PT}\ &\Leftrightarrow&\ \mbox{EC-PT+EXP}\ \Leftrightarrow\
\mbox{EC-PT+UEXP} \\
\ \ \ \ & \Leftrightarrow&\  \mbox{EC-SPT}\ \Leftrightarrow\
\mbox{EC-SPT+EXP}\ \Leftrightarrow
\ \mbox{EC-SPT+UEXP}.
\end{eqnarray*}
\item \label{appecspt}
EC-SPT+UEXP holds iff $b_j^{-1}$'s are summable and $a_j$'s are
exponentially large in~$j$, i.e.,
$$
\mbox{EC-SPT+UEXP}\ \ \Leftrightarrow\ \
B:=\sum_{j=1}^\infty\frac1{b_j}<\infty\quad\mbox{and}\quad
\alpha^*:=\liminf_{j\to\infty}\frac{\log\,a_j}j>0.
$$
Then the exponent $\tau^*$ of EC-SPT satisfies
$$
\tau^*\in \left[B,B+\min\left(B,\frac{\log\,3}{\alpha^*}\right)\right].
$$
In particular, if $\alpha^* =\infty$ then $\tau^* = B$.
\end{itemize}
\end{theorem}

Again, the conditions are the same as for the integration problem
in Theorem~\ref{thmintectract} and we have the same conditions
for $\Lambda^{\rm{all}}$ and $\Lambda^{\rm{std}}$.
The comments following Theorem~\ref{thmintectract} apply also
for approximation. We remark that the results are constructive.
The corresponding algorithms for the class $\Lambda^{\rm{all}}$
and $\Lambda^{\rm{std}}$ can be found in \cite{DKPW13}.

We want to stress that for the class $\lstd$ we obtain the results of
Theorem~\ref{thmappectract} by computing function values at grid
points with varying mesh-sizes for successive variables. Such grids
are also successfully used for multivariate integration in
\cite{DKPW13,KPW12}. This relatively simple design of sample points
should be compared with the design of (almost) optimal
sample points for analogue
problems defined over spaces of finite smoothness. In this case,
the design is much harder and requires the use of deep theory of
digital nets and low discrepancy points, see~\cite{DP10,N92}.

It is worth adding that if we use the definition~\eqref{PP13}
of WT with $\kappa>1/b_*$ then it is proved in~\cite{PP13}
that WT holds even for $a_j=a_1>0$ and $b_j=b_*>0$. Hence, the condition
$\lim_ja_j=\infty$ which is necessary and sufficient for EC-WT is now
not needed.

\section{Conclusion and Outlook}\label{secconcl}

The study of tractability with exponential convergence is a new
research subject. We presented a handful of results only for
multivariate integration and approximation problems defined over
Korobov spaces of analytic functions. Obviously, such a study should
be performed for more general multivariate problems defined over
more general spaces of $C^\infty$ or analytic functions. It would be
very much desirable to characterize multivariate problems for which
various notions of tractability with exponential convergence hold.
In this survey we presented the notions of EC-WT, EC-PT and EC-SPT.
We believe that other notions of tractability with exponential
convergence should be also studied. In fact, all notions which were
presented for tractability with respect to the pairs $(\e^{-1},s)$ can
be easily generalized and studied for the pairs $(1+\log\,\e^{-1},s)$.
In particular, the notions of EC-QPT
(exponential convergence-quasi polynomial tractability)
and EC-UWT (exponential convergence-uniform weak tractability)
are probably the first candidates for such a study.
Quasi-polynomial tract\-a\-bi\-lity was briefly mentioned in the footnote
of Section~\ref{secapp}. Uniform weak tractability generalizes
the notion of weak tractability and means that
$n(\e,s)$ is not exponential in $\e^{-\alpha}$ and $s^{\,\beta}$
for all positive $\alpha$ and $\beta$, see~\cite{S13}.

The proof technique used for EC-tractability
of integration and approximation is quite different than the
proof technique used for standard tractability. Furthermore, it
seems that some results are easier to prove for EC-tractability 
than their counterparts for the standard
tractability. In particular, optimal design of sample points seems to
be such an example. We are not sure if this holds for
other multivariate problems.

We hope that exponential convergence and tractability
will be an active research field in the future.

\end{document}